\documentclass[10pt]{article}

\input{src/project-preamble}

\usepackage[display-counters=section]{declare-style-notes}
\usepackage{declare-maths}
\usepackage{declare-text}
\usepackage{declare-tables-figures}
\usepackage[numbering-within=section]{declare-theorems-notes}
\usepackage{declare-bibliography}
\usepackage{needspace}

\hypersetup{
	pdfauthor = {Nikita Nikolaev},
	pdftitle = {Borel Summability for Nonlinear Differential Systems on Characteristic Domains},
	pdfsubject = {Singular perturbations and nonlinear differential systems in the complex domain},
	pdfkeywords = {exact perturbation theory, singular perturbation theory, Borel resummation, factorial asymptotics, nonlinear differential systems, Painleve I},
}

\newcommand{\Borel}{\frak{B}}
\newcommand{\Laplace}{\frak{L}}

\newcommand{\mm}{{\bm{m}}}

\newcommand{\MSCSubjectCode}[1]{\href{https://zbmath.org/classification/?q=cc\%3A#1}{#1}}

\title{\Large\bfseries Borel Summability for Nonlinear Differential Systems\\on Characteristic Domains}
\author{Nikita Nikolaev}
\affil{School of Mathematics, University of Edinburgh, United Kingdom}
\date{4 November 2026}

\begin{document}

\maketitle

\begin{abstract}\noindent
We investigate singularly perturbed nonlinear complex differential systems of the form $\hbar \del_x f = F (x, \hbar, f)$, where $\hbar$ is a small complex perturbation parameter.
We give geometric and analytic conditions under which the formal perturbative solution associated with a fixed regular leading-order branch is locally uniformly Borel summable.
This leads to an existence and uniqueness theorem of holomorphic solutions with prescribed perturbative expansion.
We give an explicit triangular example and a self-contained application to the singularly perturbed Painlevé I equation.
\end{abstract}

{\footnotesize
\textbf{Keywords:}
exact perturbation theory, singular perturbation theory, Borel resummation, factorial asymptotics, nonlinear differential systems, Painlevé I

\textbf{2020 MSC:}
\MSCSubjectCode{34M60} (primary),
\MSCSubjectCode{34M04},
\MSCSubjectCode{34E15},
\MSCSubjectCode{34A34}
}

\tableofcontents

\clearpage
\setcounter{section}{-1}
\section{Introduction}
\label{211214170249}
\enlargethispage{100pt}

Consider the following first-order singularly perturbed nonlinear differential system:
\begin{eqn}
	\hbar \del_x f = F (x, \hbar, f)
\fullstop{,}
\end{eqn}
Here, $x$ is a complex variable, $\hbar$ is a small complex perturbation parameter, and $f = f (x, \hbar)$ is the unknown holomorphic vector function.
The given map $F (x, \hbar, y)$ is holomorphic in all its arguments.
Suppose $F$ is a polynomial in $\hbar$, or more generally admits an asymptotic expansion $\hat{F}$ as $\hbar \to 0$ in some sector.

By expanding this equation in powers of $\hbar$, we can construct \textit{formal perturbative solutions} 
\begin{eqn}
	\hat{f} (x, \hbar) = \sum_{k=0}^\infty f^\pto{k} (x) \hbar^k
\fullstop
\end{eqn}
Once a regular leading-order solution $f^\pto{0}$ has been chosen, all higher-order coefficients of $\hat{f}$ are uniquely determined through little more than finite-dimensional matrix algebra.
However, such formal solutions generically have zero radius of convergence and therefore are not analytic objects.

\paragraph*{Main result}
Our main result (\autoref{1793810120014}) is an existence and uniqueness theorem that allows us to promote $\hat{f}$ in a canonical way to an \textit{exact perturbative solution}; i.e., a holomorphic solution $f$ whose expansion as $\hbar \to 0$ is $\hat{f}$.
In fact, we prove that $f$ is the locally uniform Borel resummation of $\hat{f}$.
This theorem assumes geometric and analytic conditions in the form of uniform factorial asymptotics of the equation coefficients on a domain which is invariant under the iterated forward-flow of a certain finite collection of holomorphic vector fields.
The latter are determined essentially by the eigenvalues of this equation's linearisation around $f^\pto{0}$ to leading-order in $\hbar$.

\paragraph*{Context}
The problem of promoting formal perturbative data to analytic data in an effective way is fundamental in what may be referred to as \textit{exact perturbation theory}, by which we mean singular perturbation theory reinforced with techniques from resurgent asymptotic analysis.
For background on Borel summability, see \cite{MR2474083,MR3495546}.
Yamazawa and Yoshino \cite{YamazawaYoshino2015} prove parametric Borel summability for semilinear systems with Euler-type spatial operators, using characteristic and convolution estimates.
Their results impose spectral conditions near the origin and give an entire-coefficient result away from it.
Here the hypotheses concern the real forward flows and bounds on their invariant domains.
Nonlinear summability at irregular singularities and the associated convolution methods are also developed in \cite{Costin1998,Ouchi2005,Ouchi2025}.

Borel summability for higher-order linear equations is another related problem.
Nemes \cite{Nemes2025} proves summability for a class of such equations, including Airy-type examples.
Fantini and Fenyes \cite{FantiniFenyes2025} develop a geometric perspective on Borel summation for differential equations using thimble integrals.
For singularly perturbed Riccati equations, Borel summation can be used to construct exact solutions \cite[\S5]{NikolaevRiccati2023}.
The Riccati reduction also gives exact WKB solutions of second-order linear equations, including the case of closed WKB trajectories \cite[Theorems~5.1 and~5.3 and Corollary~5.2]{NikolaevExactWKB2023}.
For a fully geometric treatment of the resurgence of exact WKB solutions of Schrödinger equations on Riemann surfaces, see \cite{NikolaevWKBResurgence2024}.
For the exact WKB treatment of singularly perturbed equations and Painlevé equations, see also \cite{MR2182990}.

\paragraph*{Acknowledgements}
The author thanks Dylan Allegretti, Francis Bischoff, Alberto Garc\'ia-Raboso, Marco Gualtieri, Kohei Iwaki, Olivier Marchal, Andrew Neitzke, Nicolas Orantin, Kento Osuga, and Shinji Sasaki for helpful discussions during various stages of this project.
This work was partially funded by the EPSRC Programme Grant \textit{Enhancing RNG}, as well as the European Union's Horizon 2020 Research and Innovation Programme under the Marie Skłodowska-Curie Grant Agreement No.\,101026083 (\href{https://cordis.europa.eu/project/id/101026083}{\texttt{AbQuantumSpec}}).

\paragraph*{AI Declaration}
OpenAI Codex (GPT-6.1, model identifier \texttt{gpt-6.1-sol}) was used during the preparation of this article.
Its use was limited to editorial feedback, suggestions for wording, bibliographic assistance, and proofreading.
The conceptual framework and mathematical proofs are almost entirely the author's own work, with the author retaining full accountability for the final content.

\section{Perturbative Solutions}
\label{1793786669496}

We begin by specifying the class of differential systems considered in this paper.
The main point is that the dependence on the spatial variable $x$ is assumed to be holomorphic whilst admitting a perturbative expansion in the parameter $\hbar$ which may be convergent at the origin or asymptotic in a sector.

\begin{definition}
\label{1793786900157}
Fix an integer $N \geq 1$, let $\X \subset \CC$ be a domain, let $\Y \subset \CC^N$ be an open subset, and let $\S \subset \CC$ be a sector at the origin.
A \dfn{singularly perturbed differential system}, or henceforth simply \dfn{system}, is an equation of the form
\begin{eqntag}
\label{1793787038744}
	\hbar \del_x f = F (x, \hbar, f)
\fullstop{,}
\end{eqntag}
where $F : \X \times \S \times \Y \to \CC^N$ is holomorphic.
The unknown $f = f(x,\hbar)$ is a holomorphic $\Y$-valued function on $\U \times \S'$, where $\U \subset \X$ and $\S' \subset \S \setminus \set{0}$ are domains with $0 \in \overline{\S'}$.
We suppose that $F$ admits a perturbative expansion
\begin{eqn}
	\hat{F} (x, \hbar, y) \coleq \sum_{k=0}^\infty F^\pto{k} (x,y) \hbar^k
\fullstop{,}
\end{eqn}
with coefficients holomorphic on $\X \times \Y$.
This expansion is asymptotic as $\hbar \to 0$ in $\S$, locally uniformly in $(x,y)$.
If $F$ extends holomorphically across $\hbar = 0$, this expansion is its convergent Taylor series.
\end{definition}

Here, by a sector at the origin we mean any simply connected domain $\S \subset \CC$ with $0 \in \overline{\S}$ whose opening arc of directions at $0$ is an open interval.
More accurately, a \dfn{sector} at the origin is a simply connected open subset of the real-oriented blowup of $\CC$ at $0$ in the sense of manifolds with boundary which intersects the boundary circle of directions at the origin in an open interval $\A$.

For example, for any phase $\theta$ and any $r > 0$, the \dfn{Borel disc}
\begin{eqn}
	\S (\theta, r) \coleq \set{ \hbar \in \CC^\ast ~\big|~ \Re (e^{i\theta} / \hbar) > 1 / r }
\end{eqn}
is a sector at the origin with opening $\A = (\theta - \pi/2, \theta + \pi/2)$.
Geometrically, it is a disc of diameter $r$, tangent to the origin and bisected by the direction $\theta$.
In the existence theorems, we use Borel discs bisected by the direction $\theta = 0$ and write
\begin{eqn}
	\S (r) \coleq \S (0,r) = \set{ \hbar \in \CC^\ast ~\Big|~ \Re (1 / \hbar) > 1/r }
\fullstop
\end{eqn}
An arbitrary direction can be treated by replacing the parameter $\hbar$ by $e^{i\theta} \hbar$ and the righthand side of \eqref{1793787038744} by $e^{-i\theta} F (x, e^{i\theta} \hbar, y)$.

\subsection{Formal Perturbative Solutions}

The first step is to solve the system \eqref{1793787038744} formally in powers of $\hbar$.

\begin{definition}
\label{1793788088122}
A \dfn{formal perturbative solution} on $\U \subset \X$ of a system \eqref{1793787038744} is a formal power series 
\begin{eqn}
	\hat{f} (x, \hbar) = \sum_{n = 0}^\infty f^\pto{n} (x) \hbar^n
\end{eqn}
with holomorphic coefficients on $\U$, whose leading-order term $f^\pto{0}$ takes values in $\Y$, and which satisfies the formal differential system
\begin{eqn}
	\hbar \del_x \hat{f} = \hat{F} (x, \hbar, \hat{f})
\fullstop{,}
\end{eqn}
with term-by-term differentiation in $x$.
\end{definition}

To leading order, the system \eqref{1793787038744} gives the implicit equation 
\begin{eqn}
	F^\pto{0} \big( x, y \big) = 0
\fullstop
\end{eqn}
The point is that, if this equation has a holomorphic solution $y = f^\pto{0} (x)$, which is the remit of the Holomorphic Implicit Function Theorem, then all the higher-order corrections $f^\pto{n}$ to this leading-order solution are determined uniquely by a linear algebraic equation, provided we can invert its Jacobian matrix. 

\begin{definition}
\label{220810094002}
Given a system \eqref{1793787038744}, a point $(x_0, y_0) \in \X \times \Y$ is called \dfn{regular} if
\begin{eqn}
	F^\pto{0} (x_0, y_0) = 0
\qqtext{and}
	(\del_y F^\pto{0}) (x_0, y_0) \text{ is invertible}
\fullstop
\end{eqn}
\end{definition}

\begin{proposition}[formal perturbative existence and uniqueness]
\label{211209161918}
Consider a system \eqref{1793787038744} and choose a regular point $(x_0, y_0)$.
Then $x_0$ has a neighbourhood $\U \subset \X$ on which there exists a unique formal perturbative solution $\hat{f} (x,\hbar)$ whose leading-order part satisfies $f^\pto{0} (x_0) = y_0$.
\end{proposition}

\begin{proof}
The Holomorphic Implicit Function Theorem supplies the unique leading-order solution $f^\pto{0}$ near $x_0$ with $f^\pto{0}(x_0) = y_0$.
After shrinking $\U$ if necessary, the holomorphic matrix
\begin{eqntag}
\label{1793791821475}
	J (x) \coleq (\del_y F^\pto{0}) \big( x, f^\pto{0} (x) \big)
\fullstop
\end{eqntag}
is invertible on $\U$.

Fix $n \geq 1$ and suppose that the coefficients through order $n-1$ have already been determined.
Put
\begin{eqn}
	f^{\lceil n - 1 \rceil} (x, \hbar)
		\coleq \sum_{k = 0}^{n-1} f^\pto{k} (x) \hbar^k
\fullstop{,}
\end{eqn}
and let $R_n (x)$ be the coefficient of $\hbar^n$ in
$\hat{F} \big( x, \hbar, f^{\lceil n - 1 \rceil} \big)$.
Since $\hat{f} - f^{\lceil n - 1 \rceil}$ starts with $\hbar^n f^\pto{n}$, Taylor expansion in the last argument gives
\begin{eqn}
	\hat{F} (x, \hbar, \hat{f})
		\equiv \hat{F} \big( x, \hbar, f^{\lceil n - 1 \rceil} \big)
			+ \hbar^n J f^\pto{n}
	\pmod{\hbar^{n+1}}
\fullstop
\end{eqn}
Indeed, the terms at least quadratic in $\hat{f} - f^{\lceil n - 1 \rceil}$ start at order $2n \geq n+1$.
In the linear term, the factor $\hbar^n$ means that only the constant term of $(\del_y \hat{F}) \big( x, \hbar, f^{\lceil n - 1 \rceil} \big)$ contributes at order $n$, and this constant term is $J$.

The coefficient of $\hbar^n$ in $\hbar \del_x \hat{f}$ is $\del_x f^\pto{n-1}$.
Comparing the coefficients of $\hbar^n$ on both sides of the formal differential equation gives
\begin{eqntag}
\label{220801164410}
	J f^\pto{n} = \del_x f^\pto{n-1} - R_n
\fullstop
\end{eqntag}
Since $J$ is invertible and $R_n$ depends only on the previously determined holomorphic coefficients, this equation determines a unique holomorphic coefficient $f^\pto{n}$ on $\U$.
The assertion follows by induction.
\end{proof}

We remark that for the implicit equation $F (x, \hbar, y) = 0$, a regular leading-order branch can be lifted to a Borel-resummed solution under Gevrey hypotheses on $F$ \cite[Theorem~1.1]{NikolaevImplicit2024}.

\subsection{Asymptotics of Factorial Type and Borel Resummation}

\begin{definition}
\label{1793789866981}
A holomorphic vector function $f$ on $\U \times \S (r)$ admits $\hat{f} = \sum_{k \geq 0} f^\pto{k} \hbar^k$ as its \dfn{uniform asymptotic expansion of factorial type} if there are constants $C, M > 0$ such that
\begin{eqntag}
\label{1793819997426}
	\left| f (x, \hbar) - \sum_{k = 0}^{n-1} f^\pto{k} (x) \hbar^k \right|
	\leq C M^n n! |\hbar|^n
\fullstop{,}
\end{eqntag}
for every $n \geq 0$ and every $(x, \hbar) \in \U \times \S (r)$.
In this case, we write
\begin{eqn}
	f (x, \hbar) \sim \hat{f} (x, \hbar)
\quad\text{as $\hbar \to 0$ in $\S (r)$, uniformly for $x \in \U$}
\fullstop
\end{eqn}
\end{definition}

Note that asymptotics of factorial type are also commonly called \textit{Gevrey asymptotics} of order one in the literature.

We use the following standard consequence of the Nevanlinna-Sokal Theorem \cite{MR558468}.
\begin{theorem}[Nevanlinna's uniqueness theorem]
\label{210617120300}
A holomorphic function on $\S (\theta,r)$ which satisfies \eqref{1793819997426} with zero formal expansion vanishes identically.
Consequently, the asymptotic expansion map is injective on germs of holomorphic functions with these fixed-sector factorial estimates.
\end{theorem}
Applying this theorem for each fixed $x$ also gives uniqueness for the vector-valued families considered here.

The formal Borel transform of a formal power series $\hat{f} (x, \hbar)$ is defined by
\begin{eqn}
	\hat{\phi} (x, t) =
	\hat{\Borel} [\hat{f}] (x, t)
	\coleq \sum_{n = 0}^\infty \frac{f^\pto{n+1} (x)}{n!} t^n
\fullstop
\end{eqn}
For every $\epsilon > 0$, consider the tubular neighbourhood of the positive real axis,
\begin{eqn}
	\sfSigma_\epsilon \coleq \set{ t \in \CC ~\big|~ \op{dist} (t, \RR_+) < \epsilon }
\fullstop
\end{eqn}

\begin{definition}
\label{1793790502823}
The formal series $\hat{f}$ is \dfn{uniformly Borel summable} on $\U$ if its formal Borel transform $\hat{\phi}$ converges near $t = 0$, uniformly for all $x \in \U$, admits analytic continuation $\phi$ to $\U \times \sfSigma_\epsilon$ for some $\epsilon > 0$, and there are constants $D, K > 0$ which give the following exponential estimate:
\begin{eqn}
	|\phi (x, t)| \leq D e^{K |t|}
\qtext{for all $(x, t) \in \U \times \sfSigma_\epsilon$}
\fullstop
\end{eqn}
\begingroup\setlength{\emergencystretch}{2em}
This exponential bound guarantees absolute convergence of the Laplace integral for \mbox{$\Re (1/\hbar) > K$}.
On this Borel disc, the \dfn{Borel resummation} of $\hat{f}$ is
\begin{eqn}
	s [\hat{f}] (x, \hbar)
	\coleq f^\pto{0} (x) + \Laplace [\phi] (x, \hbar)
	\coleq f^\pto{0} (x) + \int_0^\infty e^{-t/\hbar} \phi (x, t) \d{t}
\fullstop
\end{eqn}
\par\endgroup
We say that $\hat{f}$ is \dfn{locally uniformly Borel summable} on $\U$ if it is uniformly Borel summable on every relatively compact open subset $\U_0 \Subset \U$.
\end{definition}

\subsection{Normalised Systems}

Once a formal perturbative solution of \eqref{1793787038744} has been constructed, the next step is to translate the selected leading-order solution to zero and, if possible, diagonalise the Jacobian.

\begin{definition}
\label{220808113615}
A system \eqref{1793787038744} is in \dfn{normal form} if $0 \in \Y$ and
\begin{eqn}
	F^\pto{0} (x,0) = 0
\qqtext{and}
	(\del_y F^\pto{0}) (x,0)
		= \Lambda (x) \coleq 
		\diag (\lambda_1 (x),\ldots,\lambda_N (x))
\fullstop{,}
\end{eqn}
for some nonvanishing holomorphic functions $\lambda_i$, called the \dfn{eigenvalues} of the system.
\end{definition}

To reduce a given system \eqref{1793787038744} to a normal form, the following lemma gives the local change of variables explicitly whenever the Jacobian $J$ has distinct eigenvalues at a regular point.

\begin{lemma}[reduction to normal form]
\label{1793791517994}
Let $f^\pto{0}$ be a holomorphic leading-order solution on $\U \subset \X$ of a system \eqref{1793787038744}, and put $J (x) \coleq (\del_y F^\pto{0}) \big( x, f^\pto{0} (x) \big)$.
Suppose that $J$ is invertible and admits a holomorphic diagonalisation $P^{-1}JP=\Lambda$ on $\U$.
Then the change of variables 
\begin{eqn}
	f \to \tilde{f}
\qtext{given by}
	f = f^\pto{0} + P \tilde{f}
\end{eqn}
transforms the system \eqref{1793787038744} into the system $\hbar \del_x \tilde{f} = \tilde{F} (x, \hbar, \tilde{f})$ in normal form, where
\begin{eqn}
	\tilde{F} (x, \hbar, y) 
	= P^{-1} F \big(x, \hbar, f^\pto{0} + P y \big)
		- \hbar P^{-1} \del_x f^\pto{0}
		- \hbar P^{-1} (\del_x P) y
\fullstop
\end{eqn}
\end{lemma}

This reduction is well-defined wherever $f^\pto{0} (x) + P(x) \tilde{f} \in \Y$.
In particular, for every relatively compact subdomain $\U_0 \Subset \U$, there is a polydisc $\tilde{\Y}$ about zero such that $f^\pto{0} (x) + P(x)y \in \Y$ for all $(x,y) \in \U_0 \times \tilde{\Y}$.
Indeed, $f^\pto{0} (\overline{\U_0})$ is compact in $\Y$ and $P$ is bounded on $\overline{\U_0}$.
On $\U_0 \times \S \times \tilde{\Y}$, the transformed system is in normal form.
Such a reduction exists locally near every regular point where $J$ has distinct eigenvalues.

\begin{proof}
Differentiating $f = f^\pto{0} + P \tilde{f}$ and multiplying the original equation by $P^{-1}$ gives the displayed formula for $\tilde{F}$.
At leading order, $\tilde{F}^\pto{0} (x,0) = 0$ and $(\del_y \tilde{F}^\pto{0}) (x,0) = P^{-1} J P = \Lambda$.
The perturbative expansion follows by substituting $f^\pto{0} + P y$ in $\hat{F}$ and including the two displayed terms of order $\hbar$.
At a regular point, the Holomorphic Implicit Function Theorem gives $f^\pto{0}$ locally.
If the eigenvalues of $J$ are distinct at this point, then its holomorphic eigenvalue branches and local holomorphic eigenvectors give the required matrix $P$ after shrinking the neighbourhood.
\end{proof}

For a related gauge reduction near a logarithmic pole, see \cite[Theorems~3.1 and~4.1]{NikolaevInvariantFlags2026}.
Under hypotheses on the classical residue and the summation direction, a holomorphic gauge with uniform factorial asymptotics puts a rank-two system in upper-triangular form and produces an invariant line.

Let us assume that a system \eqref{1793787038744} is already in normal form.
Then we can write
\begin{eqntag}
\label{1793794417425}
	F (x, \hbar, y) = \Lambda (x) y + R (x, \hbar, y)
\end{eqntag}
where the remainder $R$ admits a perturbative expansion whose leading-order term satisfies 
\begin{eqn}
	R^\pto{0} (x,0) = 0
\qtext{and}
	(\del_y R^\pto{0}) (x,0) = 0
\fullstop
\end{eqn}
Multiplying through by the diagonal matrix $\Lambda^{-1}$ on the left leads to the more specific form of systems that we use to construct exact perturbative solutions.

\begin{definition}
\label{1793792726850}
Fix an integer $N \geq 1$, let $\X \subset \CC$ be a domain, let $\Y \subset \CC^N$ be an open neighbourhood of the origin, and let $\S \subset \CC$ be a sector at the origin.
Let $\lambda_1, \ldots, \lambda_N$ be nonvanishing holomorphic functions on $\X$.
Take the holomorphic vector fields 
\begin{eqn}
	V_i \coleq \lambda_i^{-1} \del_x
\fullstop{,}
\end{eqn}
and put $V \coleq \diag (V_1, \ldots, V_N)$.
A \dfn{normalised system} is a system of the form
\begin{eqntag}
\label{1793792948788}
	\hbar V g = g + \hbar G (x, \hbar, g)
\fullstop{,}
\end{eqntag}
where $G : \X \times \S \times \Y \to \CC^N$ is a holomorphic map which admits a perturbative expansion
\begin{eqntag}
\label{1793821450804}
	\hat{G} (x, \hbar, y) \coleq \sum_{k=0}^\infty G^\pto{k} (x,y) \hbar^k
\fullstop{,}
\end{eqntag}
with coefficients holomorphic on $\X \times \Y$.
\end{definition}

\begin{lemma}[reduction to a normalised system]
\label{1793793108372}
Consider a system \eqref{1793787038744} in normal form.
Put $H \coleq \Lambda^{-1} F$ and $a(x) \coleq -H^\pto{1}(x,0) = f^\pto{1}(x)$.
Locally on $\X$ and for all small nonzero $\hbar$, the substitution
\begin{eqn}
	f = \hbar (a + g)
\end{eqn}
transforms the system \eqref{1793787038744} in normal form to a normalised system \eqref{1793792948788} with
\begin{eqntag}
\label{1793793774071}
	G (x,\hbar,y)
		\coleq \hbar^{-2} \Big( H \big(x, \hbar, \hbar(a + y)\big) - \hbar y \Big) - Va 
\fullstop
\end{eqntag}
\end{lemma}

\begin{proof}
Near any point of $\X$, the function $a$ is bounded.
Since $\Y$ contains the origin, $\hbar(a(x)+y)$ lies in $\Y$ for $y$ in a fixed ball about zero and sufficiently small $|\hbar|$.
Examining \eqref{1793794417425}, we get
\begin{eqn}
	H^\pto{0} (x,0) = 0,
\qquad
	(\del_y H^\pto{0}) (x,0) = I,
\qquad
	H^\pto{1} (x,0) = -a(x)
\fullstop
\end{eqn}
Consequently, the terms of orders $0$ and $1$ in $H \big(x, \hbar, \hbar(a+y)\big) - \hbar y$ vanish identically.
Therefore, Taylor expansion in $y$ and perturbative expansion in $\hbar$ together give the asserted expansion for $G$.
The expansion is locally uniform in $x$ and $y$ because the original coefficient expansion is locally uniform and its derivatives in $y$ are controlled by Cauchy's formula.
In the holomorphic case, the same cancellation makes the apparent singularity at $\hbar = 0$ removable.

Direct substitution of $f = \hbar(a+g)$ into $\hbar \Lambda^{-1} \del_x f = H (x, \hbar, f)$ gives
\begin{eqn}
	\hbar^2 (Va + Vg) = \hbar g + \hbar^2 (G + Va)
\fullstop
\end{eqn}
Cancelling the common term and dividing by $\hbar$ yields \eqref{1793792948788}.
The same identities hold formally.
At order $1$ in the original system, $f^\pto{1} = a$, so the corresponding series $\hat{g} = \hbar^{-1} \hat{f} - a$ has zero constant term.
\end{proof}

Start with a system \eqref{1793787038744} on $\X \times \S$ and choose a regular point $(x_0, y_0)$ at which the leading-order Jacobian has distinct eigenvalues.
By \autoref{211209161918}, there is a holomorphic leading-order solution $f^\pto{0}$ near $x_0$ with $f^\pto{0} (x_0) = y_0$.
After shrinking to a neighbourhood $\X' \subset \X$ of $x_0$, choose a holomorphic diagonalisation $P^{-1} J P = \Lambda$ on $\X'$.
Combining \autoref{1793791517994} and \autoref{1793793108372} gives the change of variables
\begin{eqn}
	f = f^\pto{0} + \hbar P (a + g),
\qquad
	a = \Lambda^{-1} P^{-1} \big( \del_x f^\pto{0} - F^\pto{1} (x, f^\pto{0}) \big)
\fullstop{,}
\end{eqn}
which gives a normalised system \eqref{1793792948788} on $\X' \times \S'$ for a smaller sector $\S' \subset \S$ and a neighbourhood $\Y' \subset \CC^N$ of the origin.
Thus, solutions of this normalised system yield solutions of the original system over $\X' \times \S'$.

Suppose now that the eigenvalues $\lambda_i$ and the function $G$ of this normalised system admit single-valued holomorphic continuations to some domain $\X'' \subset \CC$ containing $\X'$, possibly after restricting the parameter to a smaller sector $\S'' \subseteq \S'$ and the last argument of $G$ to a smaller neighbourhood $\Y'' \subset \Y'$ of the origin.
If the eigenvalues are nonvanishing on $\X''$ and $G$ admits a perturbative expansion for all $(x,y) \in \X'' \times \Y''$, then we get a normalised system on $\X'' \times \S''$.

\section{Characteristic Domains and Exact Solutions}
\label{1793795598639}

We now consider a normalised system \eqref{1793792948788} on $\X \times \S$, where $\S = \S (r)$ is a Borel disc of some diameter $r > 0$.
Its leading-order equation is $g^\pto{0} = 0$, and the recursion of \autoref{211209161918} gives a unique formal perturbative solution $\hat{g}$ on $\X$.
In this section, we construct the corresponding exact perturbative solution $g$ as the locally uniform Borel resummation provided $G$ satisfies appropriate bounds on special subsets of $\X$.

\subsection{Common Forward-Invariant Domains}

The eigenvalues $\lambda_i$ of the normalised system are nonvanishing on $\X$, so the vector fields $V_i = \lambda_i^{-1} \del_x$ are holomorphic on $\X$.
For each $i$, write $\Phi_i^t (x)$ for the local complex flow of $V_i$.
By definition, each $\Phi_i^t$ is a single-valued holomorphic map from an open subset of $\X \times \CC$ to $\X$ which satisfies
\begin{eqn}
	\Phi_i^0 (x) = x
\qqtext{and}
	\del_t \Phi_i^t (x) = \lambda_i \big( \Phi_i^t (x) \big)^{-1}
\fullstop{,}
\end{eqn}
and obeys the local flow law wherever the relevant compositions are defined.

\begin{definition}
\label{1793807870002}
A domain $\U \subset \X$ is called a \dfn{common forward-invariant domain} for the vector fields $V_1, \ldots, V_N$ if, for each $i$, every real forward trajectory of $V_i$ originating in $\U$ is defined for all nonnegative times and remains in $\U$.
In other words, for all $x \in \U$, all $t \in \RR_+$, and every $i$, the point $\Phi^t_i (x)$ exists and belongs to $\U$.
\end{definition}

Given any domain $\U_0 \subset \X$, it may be possible to extend it to a common forward-invariant domain by collecting all points reached from $\U_0$ by a finite sequence of real forward flow segments of $V_1, \ldots, V_N$ in $\X$.
Namely, put $\U_+^\pto{0} \coleq \U_0$ and define, for every $n$,
\begin{eqn}
	\U_+^\pto{n+1}
		\coleq \Cup_{i=1}^N
		\set{ \Phi_i^t (x) ~\Big|~ x \in \U_+^\pto{n},~ t \geq 0,~
			\Phi_i^s (x) \text{ exists in $\X$ for all $s \in [0,t]$} }
\fullstop
\end{eqn}
These sets form a nested increasing chain $\U_+^\pto{n} \subset \U_+^\pto{n+1}$, and we let
\begin{eqn}
	\U_+ \coleq \Cup_{n = 0}^\infty \U_+^\pto{n}
\fullstop
\end{eqn}
The set $\U_+$ is open because, for each fixed real time $t$, the flow map $x \mapsto \Phi_i^t (x)$ is locally biholomorphic on the open set of initial points whose trajectories exist through time $t$.
It is connected because every point of $\U_+$ can be joined to $\U_0$ along the successive flow segments by which it was reached.
Every common forward-invariant domain containing $\U_0$ contains $\U_+$.
If every forward trajectory of each vector field $V_i$ starting in $\U_+$ exists for all nonnegative times, then $\U_+$ itself is forward-invariant, because any such trajectory can be appended to the succession that reached its initial point.
In this case, $\U_+$ is the smallest common forward-invariant domain containing $\U_0$.

\subsection{Existence and Uniqueness of Exact Perturbative Solutions}

\enlargethispage{20pt}

The main result of this paper is the following theorem.

\begin{theorem}[exact existence and uniqueness]
\label{1793810120014}
Consider a normalised system \eqref{1793792948788} on the domain $\X \times \S$ where $\S = \S (r)$ is a Borel disc.
Suppose $\U \subset \X$ is any common forward-invariant domain for the vector fields $V_1, \ldots, V_N$ such that the holomorphic map $G = G (x,\hbar,y)$ admits its perturbative expansion $\hat{G} = \hat{G} (x,\hbar,y)$ as uniform asymptotic expansion of factorial type, uniformly for all $x \in \U$ and $y$ in a ball around the origin.

Then the formal perturbative solution $\hat{g} = \hat{g} (x, \hbar)$ is locally uniformly Borel summable on $\U$.
In fact, there is some Borel disc $\S_0 = \S (r_0)$ of diameter $r_0 < r$ such that the Borel resummation $g (x, \hbar) \coleq s [ \hat{g} ] (x, \hbar)$ defines a bounded holomorphic solution of \eqref{1793792948788} on $\U \times \S_0$.
Thus, for every compactly contained open subset $\U_0 \subset \U$, the Borel resummation $g$ is the unique holomorphic solution on $\U_0 \times \S_0$ admitting $\hat{g}$ as its uniform asymptotic expansion of factorial type.
\end{theorem}

Let us spell out the asymptotic hypothesis on $G$ more explicitly: there is a sufficiently small polydisc $\sfDelta = \set{ y \in \CC^N ~:~ |y| < R} \subset \Y$ and constants $C, M > 0$ such that
\begin{eqntag}
\label{1793812150781}
	\left| G (x, \hbar, y) - \sum_{k=0}^{n-1} G^\pto{k} (x, y) \hbar^k \right|
		\leq C M^n n! |\hbar|^n
\fullstop{,}
\end{eqntag}
for all $n \geq 0$ and all $(x, \hbar, y) \in \U \times \S \times \sfDelta$.

\begin{proof}
\caps{Step 1. The characteristic integral.}
Fix $\rho > 0$ so that the closed polydisc of radius $\rho$ lies in $\sfDelta$. The order-zero estimate bounds $G$ on
$\U \times \S \times \sfDelta$ by a constant $C$.
Cauchy's estimates in $y$ give a uniform Lipschitz constant $L$ for $G$ on the closed polydisc of radius $\rho$, in the maximum norm.
For any $r_0 < r$, set $\S_0 \coleq \S (r_0)$, and consider the integral operator
\begin{eqntag}
\label{1793820360961}
	(T g)_i (x,\hbar)
	\coleq -\int_0^\infty e^{-u/\hbar}
		G_i \big( \Phi_i^u (x), \hbar, g (\Phi_i^u (x),\hbar) \big) \, \d{u}
\fullstop
\end{eqntag}
The real forward flow stays in $\U$, so all evaluations are defined for a holomorphic $g$ of supremum norm at most $\rho$ on $\U \times \S_0$.
For each finite real $u$, the flow $\Phi_i^u (x)$ is holomorphic in the initial point $x$.
Since $\Re (1/\hbar) > 1/r_0$ on $\S_0$, the bound $C e^{-u/r_0}$ gives locally uniform convergence of the integral and holomorphy in $(x,\hbar)$.
It satisfies 
\begin{eqn}
	\|T g\|_\infty \leq C r_0
\qtext{and}
	\|T g-T \tilde{g}\|_\infty \leq Lr_0\|g-\tilde{g}\|_\infty
\fullstop
\end{eqn}
Choose $r_0$ so that $Cr_0 \leq \rho$ and $Lr_0 < 1$.
The Contraction Mapping Theorem on the closed ball of radius $\rho$ in the space of bounded holomorphic functions gives a fixed point $g$ defined on $\U \times \S_0$.

Evaluating \eqref{1793820360961} at $\Phi_i^s (x)$ and using the flow law gives
\begin{eqn}
	g_i (\Phi_i^s (x),\hbar)
		= - e^{s/\hbar}\int_s^\infty e^{-u/\hbar}
			G_i \big( \Phi_i^u (x),\hbar,g(\Phi_i^u(x),\hbar) \big) \, \d{u}
\fullstop
\end{eqn}
Differentiating at $s=0$ yields $V_i g_i = \hbar^{-1} g_i + G_i(x,\hbar,g)$, as required.

\caps{Step 2. Uniform factorial estimates.}
Fix a relatively compact open subset $\U_0 \Subset \U$.
For each $k \geq 0$, divide the remainder in \eqref{1793812150781} at order $k$ by $\hbar^k$.
The expansion through degree $k$ shows that the quotient tends to $G^\pto{k} (x,y)$ as $\hbar \to 0$.
Taking the limit gives $|G^\pto{k} (x,y)| \leq CM^k k!$ on $\U \times \sfDelta$.
For each $n \geq 1$, put
\begin{eqn}
	P_n (x,z,y) \coleq \sum_{k=0}^{n-1} G^\pto{k}(x,y) z^k
\fullstop
\end{eqn}
Take $d>0$ such that the closed $d$-neighbourhood of $\U_0$ is contained in $\U$.
For $0 \leq j < n$, put
\begin{eqn}
	\D_j \coleq \set{ x \in \CC ~\big|~ \op{dist} (x,\U_0) < d(1-j/n) }
\qqtext{and set}
	\D_n \coleq \U_0
\fullstop
\end{eqn}
For every bounded holomorphic vector function $v$ on $\D_j$, Cauchy's estimate in $x$ gives
\begin{eqn}
	\|Vv\|_{\D_{j+1}} \leq A n \|v\|_{\D_j}
\fullstop{,}
\end{eqn}
where $A$ is independent of $n,j$ since the $\lambda_i^{-1}$ are bounded on the closed $d$-neighbourhood of $\U_0$.
Here and below, the norms are spatial supremum norms at a fixed value of the parameter.

Choose $c>0$ so small that $Mc \leq 1/2$,~ $c(A\rho+2C) \leq \rho$, and $c(A+L) \leq 1/2$.
Since $k! \leq n^k$ for $k<n$, the polynomial $P_n$ is bounded by $2C$ for $|z|<c/n$ and $y\in\sfDelta$.
Define successively
\begin{eqn}
	u_0 \coleq 0
\qqtext{and}
	u_{j+1} \coleq z V u_j - z P_n (x,z,u_j)
	\qquad (0\leq j<n)
\fullstop
\end{eqn}
Induction using the spatial Cauchy estimate gives $\|u_j\|_{\D_j}\leq \rho$ on $|z|<c/n$.
Each iteration gives one more coefficient of the formal solution of
$\hat g=zV\hat g-z\hat G(x,z,\hat g)$, since substitution on the righthand side raises the order of an error by one.
Truncating $\hat G$ in $P_n$ first affects the iteration at order $n+1$.
Thus, on $\U_0$ the Taylor coefficients of $u_n$ through order $n$ are the coefficients $g^\pto{k}$ of $\hat g$.
Cauchy's estimates in $z$ give
\begin{eqn}
	\left\|u_n (x,\hbar)-\sum_{k=0}^{n-1} g^\pto{k}(x)\hbar^k\right\|_{\U_0}
	\leq 2\rho (n/c)^n|\hbar|^n
	\qquad \text{if $|\hbar|\leq c/(2n)$}
\fullstop{,}
\end{eqn}
and, using $n=k$ for $k\geq 1$, give the bound
\begin{eqn}
	\|g^\pto{k}\|_{\U_0}\leq \rho(k/c)^k \leq \rho(e/c)^k k!
\fullstop
\end{eqn}

For $\hbar\in\S_0$ with $|\hbar|\leq c/(2n)$, put 
\begin{eqn}
	E_j \coleq \|g(x,\hbar)-u_j(x,\hbar)\|_{\D_j}
\fullstop
\end{eqn}
Subtract the iteration from $g = \hbar Vg-\hbar G (x,\hbar,g)$ and apply the spatial Cauchy estimate, the Lipschitz bound, and the remainder estimate for $G-P_n$.
This gives
\begin{eqn}
	E_{j+1} \leq (An+L)|\hbar| E_j + C M^n n!|\hbar|^{n+1}
\fullstop
\end{eqn}
Since $(An+L)|\hbar|\leq 1/2$ and $E_0\leq\rho$, we obtain
\begin{eqn}
	E_n \leq \rho \big( (A+L)n|\hbar| \big)^n + 2 C M^n n!|\hbar|^{n+1}
\fullstop
\end{eqn}
Together with the Taylor remainder of $u_n$ and $n^n\leq e^n n!$, this gives the desired factorial estimate in this range.

If $|\hbar|>c/(2n)$, then for $0\leq k<n$ we have
\begin{eqn}
	k!|\hbar|^k
		\leq (2e/c)^{n-k}\, n!|\hbar|^n
\fullstop{,}
\end{eqn}
since the geometric mean of $k+1,\ldots,n$ is at least $(n!)^{1/n}\geq n/e$.
Using this estimate with the coefficient bounds and $\|g\|_{\U_0}\leq \rho$ gives the factorial remainder bound in this range as well, after increasing the constants.
The order-zero estimate follows from the boundedness of $g$.
We have therefore obtained \eqref{1793819997426} on the whole product $\U_0\times\S_0$.

\enlargethispage{30pt}

\caps{Step 3. Borel summation.}
This step uses only the holomorphy of the solution on $\U_0 \times \S_0$ and the uniform factorial estimates there.
Apply the Nevanlinna-Sokal Theorem \cite{MR558468} componentwise to these estimates.
The series 
\begin{eqn}
	\hat{B} (x,t) \coleq \sum_{k = 0}^\infty g^\pto{k} (x) \frac{t^k}{k!}
\end{eqn}
converges on a sufficiently small disc about $t = 0$, uniformly for $x \in \U_0$ by the coefficient bounds.
The Nevanlinna-Sokal Theorem gives its continuation $B (x,t)$ to $\U_0 \times \sfSigma_\epsilon$ with the bound $|B (x,t)| \leq D e^{|t|/r_0}$.
The width $\epsilon$ and the constant $D$ are uniform in $x$ because the scalar estimates depend only on $r_0$ and the constants in \eqref{1793819997426}.
Taylor expansion on successive overlapping discs along the positive real axis propagates holomorphy in $x$ from the convergent series at $t = 0$, with locally uniform convergence given by these bounds.

The Borel transform of $\hat{g}$ is $\phi = \del_t B$.
Cauchy's estimate gives a bound of the same exponential type for $\phi$ on a thinner tube.
On the whole Borel disc $\S_0$, we have $\Re (1/\hbar) > 1/r_0$, so the Laplace representation and integration by parts give
\begin{eqn}
	g (x,\hbar)
		= \frac{1}{\hbar} \int_0^\infty e^{-t/\hbar} B (x,t) \d{t}
		= g^\pto{0} (x) + \int_0^\infty e^{-t/\hbar} \phi (x,t) \d{t}
		= s [\hat{g}] (x,\hbar)
\fullstop
\end{eqn}
Finally, \autoref{210617120300} applied for each $x \in \U_0$ gives uniqueness with the prescribed uniform factorial expansion.
\end{proof}

\subsection{Domains Adapted to the Coupling}
\label{1793823067016}

Looking back over the proof of \autoref{1793810120014}, we can notice that the integral for $g_i$ uses only the components on which $G_i$ depends, evaluated along the forward flow of $V_i$.
We can therefore choose a spatial domain for each component, with inclusions determined by these dependencies.

\begin{theorem}[domains adapted to the coupling]
\label{1793816223566}
Consider a normalised system \eqref{1793792948788} on the domain $\X \times \S$ where $\S = \S (r)$ is a Borel disc.
For each $i$, choose a set $\E_i \subset \set{1, \ldots, N}$ such that the component $G_i$ depends only on the variables $y_j$ with $j \in \E_i$.
Suppose $\U_1, \ldots, \U_N \subset \X$ are open sets with the following properties for each $i$:
\begin{enumerate}
\item $\U_i$ is a forward-invariant domain under the vector field $V_i$;
\item $\U_i \subseteq \U_j$ whenever $j \in \E_i$;
\item $G_i$ admits its perturbative expansion $\hat{G}_i$ as uniform asymptotic expansion of factorial type, uniformly for all $x \in \U_i$ and $y$ in a ball around the origin.
\end{enumerate}

Then each component $\hat{g}_i$ of the formal perturbative solution $\hat{g}$ is locally uniformly Borel summable on $\U_i$.
In fact, there is some Borel disc $\S_0 = \S (r_0)$ of diameter $r_0 < r$ such that each Borel resummation $g_i (x, \hbar) \coleq s [\hat{g}_i] (x,\hbar)$ defines a bounded holomorphic function on $\U_i \times \S_0$, and the vector $g \coleq (g_1, \ldots, g_N)$ is a holomorphic solution of \eqref{1793792948788} on $\U \times \S_0$, where $\U \coleq \U_1 \cap \cdots \cap \U_N$.
Thus, for every compactly contained open subset $\U_0 \subset \U$, the Borel resummation $g$ is the unique holomorphic solution on $\U_0 \times \S_0$ admitting $\hat{g}$ as its uniform asymptotic expansion of factorial type.
\end{theorem}

Note that the total intersection $\U$ may in principle be disconnected or empty.
Regardless, the componentwise conclusions concern the individual domains $\U_i$ in any case.

\begin{proof}
Choose a polydisc $\sfDelta \subset \Y$ about zero on which the factorial estimates hold for all components.
By the first hypothesis, the point $\Phi_i^u (x)$ in the characteristic integral \eqref{1793820360961} for $g_i$ is defined and belongs to $\U_i$ for every $x \in \U_i$ and every $u \geq 0$.
By the domain inclusions, the required values of $g_j$ are defined there for every $j \in \E_i$.
The variables with indices outside $\E_i$ may be set to zero in the argument of $G_i$.

Take $\S_0 \coleq \S (r_0)$ with $0 < r_0 < r$.
The integral operator acts on the product of the spaces of bounded holomorphic functions on $\U_i \times \S_0$, equipped with the maximum of the supremum norms.
As in Step~1 of the proof of \autoref{1793810120014}, the order-zero factorial bounds and Cauchy's estimates in $y$ give a contraction on a closed ball for sufficiently small $r_0$, taking the maximum of the finitely many constants for the different components.
Differentiating along the flow of $V_i$ gives the $i$th equation on $\U_i \times \S_0$.
In particular, the vector $g$ of the restrictions solves \eqref{1793792948788} on $\U \times \S_0$.

Fix $i$, and let $\J_i$ be the smallest set of indices containing $i$ and containing $\E_j$ whenever it contains $j$.
Repeated use of the domain inclusions gives $\U_i \subseteq \U_j$ for every $j \in \J_i$.
The components indexed by $\J_i$ form a subsystem on $\U_i$ because each of its equations uses only variables in $\J_i$.
The corresponding restrictions of $g_j$ solve this subsystem, and its formal solution is the tuple of the $\hat{g}_j$ for $j \in \J_i$, by the formal recursion.

For every relatively compact open subset $\U_0 \Subset \U_i$, the boundedness of these $g_j$ and the uniform factorial estimates for the corresponding $G_j$ hold on a relatively compact neighbourhood of $\U_0$ in $\U_i$.
The local argument of Step~2 in the proof of \autoref{1793810120014} therefore gives \eqref{1793819997426} on $\U_0 \times \S_0$ for the components of the subsystem.
Step~3 gives uniform Borel summability there and identifies $g_i$ with $s [\hat{g}_i]$.
Since $\S_0$ was chosen before fixing $i$ or $\U_0$, this proves the componentwise assertion on the same sector for every $i$.
On each relatively compact open subset of $\U$, the restrictions of these resummations give $g = s [\hat{g}]$.
Uniqueness with the prescribed uniform factorial expansion follows componentwise from \autoref{210617120300}.
\end{proof}

A chain of dependencies gives a chain of domain inclusions.
In particular, the domains of components lying in a cycle of dependencies coincide.
For a triangular system the inclusions can be strict.

\begin{example}[a triangular nonlinear system]
\label{1793823079940}
Consider the normalised system on $\X = \CC$
\begin{eqntag}
\label{1793823085643}
	\hbar \del_x g_1 = g_1 + \hbar \big( e^{ix} + g_2 \big)
\qqtext{and}
	i\hbar \del_x g_2 = g_2 + \hbar \big( e^{-x} + g_2^2 \big)
\fullstop
\end{eqntag}
Take any Borel disc $\S = \S (r)$ and put $\E_1 = \E_2 = \set{2}$.
Fix $d > 0$ and take
\begin{eqn}
	\U_1 \coleq \set{ x \in \CC ~\big|~ \Re x > d,\ \Im x > d }
\qqtext{and}
	\U_2 \coleq \set{ x \in \CC ~\big|~ \Re x > d }
\fullstop
\end{eqn}
The real forward flows of $V_1 = \del_x$ and $V_2 = i\del_x$ are $x \mapsto x+u$ and $x \mapsto x+iu$, respectively.
The flow of $V_1$ preserves $\U_1$ and the flow of $V_2$ preserves $\U_2$.
Together with $\U_1 \subseteq \U_2$, this verifies the flow and inclusion hypotheses of \autoref{1793816223566}.
Both $G_i$ are independent of $\hbar$, and $G_1$ is bounded over $\U_1$ and $G_2$ over $\U_2$ for $y$ in any fixed bounded polydisc.
Their factorial estimates therefore reduce to these bounds.
By \autoref{1793816223566}, the resummations of $\hat{g}_1$ on $\U_1$ and of $\hat{g}_2$ on the larger domain $\U_2$ are defined on the same Borel disc $\S_0$.
Each resummation is locally uniform in its spatial domain.

The full system satisfies the hypotheses of \autoref{1793810120014} on $\U_1$.
On $\U_2$, however, $|G_1 (d+1-iu,\hbar,0)| = e^u$ for $u \geq 0$, so the uniform bound for $G_1$ fails there.
\end{example}

\subsection{An Analytic Criterion for the Original System}

For a system in normal form, uniform holomorphic bounds near $\hbar = 0$ give a direct way to verify \eqref{1793812150781} after the substitution of \autoref{1793793108372}.
When $H$ is a polynomial, boundedness of its coefficients gives the required holomorphic bound.

\begin{proposition}[analytic or polynomial coefficient criterion]
\label{1793818664246}
Consider a system \eqref{1793787038744} in normal form on $\X \times \S$, where $\S = \S (r)$ is a Borel disc.
Introduce
\begin{eqn}
	H \coleq \Lambda^{-1} F
\qtext{and}
	a(x) \coleq -H^\pto{1} (x,0)
\fullstop
\end{eqn}
Suppose $\U \subseteq \X$ is a common forward-invariant domain for the vector fields $V_i$ on which $Va$ is uniformly bounded.
Furthermore, suppose that there is a disc $\DD = \set{ |\hbar| < \rho } \subset \CC$ and a polydisc $\sfDelta = \set{ |y| < R } \subset \Y$ such that $H$ extends to a holomorphic and uniformly bounded function on $\U \times \DD \times \sfDelta$.
In particular, if
\begin{eqn}
	H (x, \hbar, y) = \sum_{k,\mm} H_{k,\mm} (x) \hbar^k y^\mm
\fullstop
\end{eqn}
is a polynomial in $\hbar$ and $y$, it is sufficient to assume that $Va$ and the coefficients $H_{k,\mm}$ are uniformly bounded on $\U$.

Then the formal solution $\hat{f}$ with $f^\pto{0} = 0$ is locally uniformly Borel summable on $\U$.
In fact, there is some Borel disc $\S_0 = \S (r_0)$ of diameter $r_0 < r$ such that the Borel resummation $f (x,\hbar) = s [\hat{f}] (x,\hbar)$ defines a holomorphic solution of \eqref{1793787038744} on $\U \times \S_0$.
Thus, for every compactly contained open subset $\U_0 \subset \U$, the Borel resummation $f$ is the unique holomorphic solution on $\U_0 \times \S_0$ admitting $\hat{f}$ as its uniform asymptotic expansion of factorial type.
\end{proposition}

\begin{proof}
If $H$ is polynomial in $\hbar$ and $y$ with uniformly bounded coefficients on $\U$, its finite sum gives a holomorphic extension which is uniformly bounded on $\U \times \DD \times \sfDelta$ for any fixed finite radii with $\sfDelta \subset \Y$.
It therefore suffices to prove the assertion under the holomorphic extension hypothesis.

Let $M$ be a uniform bound for $H$ on $\U \times \DD \times \sfDelta$.
By uniqueness of the asymptotic coefficients, $H^\pto{1} (x,0) = (\del_\hbar H) (x,0,0)$ for the holomorphic extension.
Cauchy's estimate gives $|a| \leq M/\rho$ on $\U$.
Choose $R_1 > 0$ and $0 < \rho_1 < \rho$ with $\rho_1 (M/\rho + R_1) < R$, and put $\sfDelta_1 \coleq \set{ y \in \CC^N ~:~ |y| < R_1 }$.
Then the function
\begin{eqn}
	K (x,\hbar,y) \coleq H \big( x,\hbar,\hbar(a(x)+y) \big) - \hbar y
\fullstop
\end{eqn}
is holomorphic and bounded by $M + \rho_1 R_1$ for $x \in \U$, $|\hbar| < \rho_1$, and $y \in \sfDelta_1$.
As in the proof of \autoref{1793793108372}, it vanishes to order at least two in $\hbar$.
Schwarz's lemma therefore gives
\begin{eqn}
	\big| \hbar^{-2} K (x,\hbar,y) \big|
		\leq \frac{M + \rho_1 R_1}{\rho_1^2}
\fullstop
\end{eqn}
Since $Va$ is bounded, the function $G = \hbar^{-2} K - Va$ is holomorphic and uniformly bounded on the same product, including at $\hbar=0$.
Cauchy's estimates in $\hbar$ give a uniform bound of the form $C_G \rho_1^{-k}$ for its Taylor coefficients.
For $0 < r_1 < \min\{r,\rho_1\}$, the Taylor remainder of order $n$ is therefore bounded by
\begin{eqn}
	\frac{C_G}{1-r_1/\rho_1} \rho_1^{-n} |\hbar|^n
\fullstop{.}
\end{eqn}
This gives \eqref{1793812150781} on $\U \times \S (r_1) \times \sfDelta_1$.

Apply \autoref{1793810120014} to this normalised system on $\U \times \S (r_1)$.
The solution $g$ takes values in $\sfDelta_1$ by the fixed-point construction, and has a uniform asymptotic expansion of factorial type on each $\U_0 \times \S_0$ with $\U_0 \Subset \U$.
Our choice of $\rho_1$ ensures that $f = \hbar (a+g)$ takes values in $\sfDelta$, so it solves the given system by \autoref{1793793108372}.
The corresponding formal expansion is $\hat{f} = \hbar (a+\hat{g})$.
For $n \geq 2$, the remainder after truncating $\hat{f}$ at degree $n-1$ is $\hbar$ times the remainder for $g$ after truncating $\hat{g}$ at degree $n-2$.
The estimates at orders zero and one follow from the boundedness of $a$ and $g$.
Thus $f$ admits $\hat{f}$ as a uniform asymptotic expansion of factorial type on each $\U_0 \times \S_0$.
Step~3 of the proof of \autoref{1793810120014} now identifies $f$ with the Borel resummation of $\hat{f}$ and gives uniqueness on each such product.
\end{proof}

The bound on $Va$ also follows if $H$ is holomorphic and uniformly bounded on $\D \times \DD \times \sfDelta$ for a larger spatial domain $\D$ and the complex-time flows $\Phi_i^u (x)$ stay in $\D$ for all $x \in \U$ and $|u| < \tau$, with $\tau > 0$ independent of $x$.
Indeed, $a$ is bounded on $\D$, and Cauchy's estimate in $u$ for $a_i \big( \Phi_i^u (x) \big)$ gives the required bound on $V_i a_i$.

\begin{example}[analytic righthand side]
\label{1793820776288}
Consider the system on $\X \times \S$ with $\X = \CC$ and $\S = \S (r)$ for $0 < r < 1$
\begin{eqns}
\label{1793820782115}
	\hbar \del_x f_1
		&= f_1 + e^{-x} \big( e^{f_2}-1-f_2 \big)
			+ \frac{\hbar e^{-x}}{1-\hbar}
\fullstop{,}
\\
	i\hbar \del_x f_2
		&= f_2 + e^{-x} \big( e^{f_1}-1-f_1 \big)
			+ \frac{\hbar e^{-x}}{1+\hbar}
\fullstop
\end{eqns}
The system is in normal form with $\Lambda = \diag (1,-i)$, so $V_1 = \del_x$ and $V_2 = i \del_x$.
For $d>0$, take
\begin{eqn}
	\U \coleq \set{ x \in \CC ~\big|~ \Re x>d,\ \Im x>d }
\fullstop
\end{eqn}
The real forward flows $x \mapsto x+u$ and $x \mapsto x+iu$, for $u \geq 0$, preserve $\U$.
Writing the system as $\hbar \del_x f=\Lambda H(x, \hbar, f)$, we have
\begin{eqn}
	H_i(x, \hbar, y)
		=y_i+e^{-x}\big(e^{y_{3-i}}-1-y_{3-i}\big)
			+\frac{\hbar e^{-x}}{1-(-1)^{i-1}\hbar}
\qquad(i=1,2)
\fullstop
\end{eqn}
For every $0 < \rho < 1$ and every bounded polydisc $\sfDelta$ about zero, this function is holomorphic and uniformly bounded on $\U \times \DD \times \sfDelta$, where $\DD = \set{ |\hbar| < \rho }$.
Its first correction is $f^\pto{1}=(-e^{-x},-e^{-x})$, so both $V_i f_i^\pto{1}$ are bounded on $\U$.
Thus \autoref{1793818664246} gives uniform Borel summability on every $\U_0 \Subset \U$.
\end{example}

\subsection{Singularities and Growth along the Flows}
\label{1793823099624}

The next two examples illustrate the geometric and boundedness hypotheses of \autoref{1793810120014}.
In the first, a singularity of a coefficient is reached by a characteristic and produces a pole of the Borel transform.

\begin{example}[a characteristic singularity]
\label{1793820791832}
Consider the scalar equation on $\X = \CC \setminus \set{1}$
\begin{eqntag}
\label{1793823115101}
	\hbar \del_x g = g + \frac{\hbar}{1-x}
\fullstop
\end{eqntag}
Near $x=0$, its righthand side is holomorphic in all variables and uniformly bounded on every sufficiently small fixed product of discs.
Its leading-order solution is zero and its leading-order Jacobian is $1$.
The selected formal solution is
\begin{eqn}
	\hat{g} (x, \hbar)
		= -\sum_{n=1}^\infty \frac{(n-1)!}{(1-x)^n}\hbar^n
\fullstop
\end{eqn}
Indeed, the coefficient recursion gives $g^\pto{1}=-(1-x)^{-1}$ and $g^\pto{n+1}=\del_x g^\pto{n}$ for $n\geq1$.
Its formal Borel transform converges near $t=0$ to
\begin{eqntag}
\label{1793822915787}
	\hat{\Borel}[\hat{g}](x, t)
		= -\sum_{n=0}^\infty\frac{t^n}{(1-x)^{n+1}}
		= -\frac{1}{1-x-t}
\fullstop
\end{eqntag}
At $x=0$, this function has a pole at $t=1$ on the positive real ray, so $\hat{g}$ is not Borel summable along the positive real ray.
The ambient straight-line trajectory from $0$ meets the singularity $x=1$ at time one.
The forward trajectory in $\X = \CC \setminus \set{1}$ therefore exists only for $0 \leq t < 1$.
In particular, no common forward-invariant domain in $\X$ contains $0$.

This also shows why each component domain in \autoref{1793816223566} must be forward-invariant under its own vector field.
Take $\E_1 = \varnothing$ and $\U_1 = \set{ x \in \CC ~:~ |x| < 1/2 }$.
The conditions involving $\E_1$ are vacuous, and $G_1 = (1-x)^{-1}$ is uniformly bounded and independent of $\hbar$ on this domain.
However, $\U_1$ fails the first hypothesis because it is not forward-invariant under $V_1 = \del_x$.
Nevertheless, \eqref{1793822915787} has a pole on the positive real ray at $x=0$.
\end{example}

A related link between characteristic trajectories and Borel singularities occurs in exact WKB analysis.
For Schrödinger equations on Riemann surfaces, the spectral cover describes the endless analytic continuation of the Borel transform, and critical trajectories determine its visible singularities \cite[Propositions~4.2 and~4.3]{NikolaevWKBResurgence2024}.

In the next example, the coefficients $a$ and $b$ are entire and decay along each of the two characteristic flows.
Composing the flows nevertheless produces superexponential growth of the Borel transform along the positive real ray.

\begin{example}[mixed characteristic growth]
\label{1793822907048}
Put
\begin{eqn}
	a (x) \coleq e^{-x^4},
\qqquad
	b (x) \coleq e^{cx^4},
\qqtext{where}
	c \coleq \frac{-37 + 55i}{8}
\fullstop{,}
\end{eqn}
and consider the following normalised system on $\X = \CC$:
\begin{eqntag}
\label{1793822896319}
	\hbar \del_x g_1 = g_1 + \hbar b (x) g_2
\qqtext{and}
	i \hbar \del_x g_2 = g_2 + i \hbar a (x)
\fullstop
\end{eqntag}
Here, $V_1 = \del_x$,~ $V_2 = i\del_x$,~ $G_1(x, \hbar, y) = b(x)y_2$,~ $G_2(x, \hbar, y) = i a(x)$.
The real forward trajectories are $x_0 + t$ and $x_0 + it$, both of which are defined for all $t \geq 0$.
For any compact subset $\K \subset \CC$, any $\eta \in \set{1,i}$, and any $\gamma \in \set{-1,c}$, we have
\begin{eqn}
	\Re \big( \gamma (x_0 + \eta t)^4 \big)
		= (\Re \gamma) t^4 + O_\K (t^3)
\qqqquad
	\forall x_0 \in \K
\fullstop
\end{eqn}
Since $\Re \gamma < 0$, both $a$ and $b$ decay along both families of trajectories, uniformly for all initial points in $\K$.
Consequently, $G$ is uniformly bounded along the two trajectories from each $x_0 \in \K$, for $y$ in any fixed bounded polydisc.
Since $G$ is independent of $\hbar$, it has uniform factorial asymptotics along these trajectories.
However, any common forward-invariant domain containing $0$ contains the points $(1+i)t$ for all $t \geq 0$.
At these points $a ((1+i)t) = e^{4t^4}$, so $G_2$ is unbounded.

In fact, at $x=0$ the formal perturbative solution of \eqref{1793822896319} is not Borel summable along the positive real ray.
The formal Borel transform $\hat{\phi} = \hat{\Borel} [\hat{g}]$ has the following entire continuation:
\begin{eqnstag}
\label{1793822888010}
	\phi_2 (x, t)
		&= -i a (x + it)
\fullstop{,}
\\
\label{1793822854821}
	\phi_1 (x, t)
		&= i \int_0^t b (x + t - u) a (x + t - u + iu) \d{u}
\fullstop
\end{eqnstag}
Indeed, the formal Borel transform gives $(V_2-\del_t)\phi_2=0$ and $(V_1-\del_t)\phi_1=b\phi_2$, with initial values $\phi_2 (x, 0) = -i a (x)$ and $\phi_1 (x, 0) = 0$.
Integrating these equations along the characteristics gives the two formulas.
Putting $u = st$ writes the integral over $[0,1]$ with an integrand entire in $(x, t)$.

At $x = 0$ and $t = R > 0$, put
\begin{eqn}
	E (s) \coleq c (1-s)^4 - (1-s+is)^4
\fullstop
\end{eqn}
Then
\begin{eqn}
	\phi_1 (0,R) = iR \int_0^1 e^{R^4 E (s)} \d{s}
\fullstop
\end{eqn}
Direct calculation at $s_0 = 3/5$ gives
\begin{eqn}
	E' (s_0) = 0,
\qqquad
	E (s_0) = \frac{9 + 46i}{125},
\qqquad
	E'' (s_0) = \frac{-102 + 42i}{5}
\fullstop
\end{eqn}
Moreover, for all $0 \leq s \leq 1$,
\begin{eqntag}
\label{1793822878413}
	\Re E (s_0) - \Re E (s)
		= (s-3/5)^2 \frac{25s^2 - 390s + 633}{40}
		\geq \frac{67}{10} (s-3/5)^2
\fullstop
\end{eqntag}
Set $v = R^2 (s-s_0)$ and factor out $e^{R^4 E (s_0)}$.
Extend the rescaled integrand by zero outside $[-R^2s_0,R^2(1-s_0)]$.
The extended integrands are bounded by $e^{-67v^2/10}$ and converge pointwise on $\RR$ to $e^{E'' (s_0)v^2/2}$.
Dominated convergence therefore gives
\begin{eqntag}
\label{1793822867672}
	\phi_1 (0,R)
		= \frac{i}{R} \exp \left( \frac{9 + 46i}{125} R^4 \right)
		\left( \sqrt{\frac{5\pi}{51 - 21i}} + o(1) \right)
\fullstop
\end{eqntag}
Since the square-root factor is nonzero, $|\phi_1 (0,R)| \sim C R^{-1} e^{9R^4/125}$ for a constant $C>0$.
This violates every at-most-exponential bound required for Borel summability.

In \eqref{1793822854821}, the function $a$ is evaluated at $x + (R-u) + iu$ when $t = R$.
This point is reached by first following $V_1$ for time $R-u$ and then $V_2$ for time $u$.
The bounds along the two separate trajectories from $x$ do not control these evaluations.
\end{example}

\section{The Painlevé I Equation}
\label{1793823134814}

We now apply \autoref{1793810120014} to the singularly perturbed Painlevé I equation
\begin{eqntag}
\label{1793822831332}
	\hbar^2 \del_x^2 q = 6q^2 + x
\fullstop
\end{eqntag}
Put
\begin{eqn}
	\kappa_\pm \coleq \pm \frac{i}{\sqrt6},
\qqquad
	\rho \coleq \sqrt{12/\sqrt6},
\qqquad
	\varepsilon_+ \coleq \rho e^{\pi i/4},
\qqquad
	\varepsilon_- \coleq \rho e^{3\pi i/4}
\fullstop
\end{eqn}
Thus, $\varepsilon_\pm^2 = 12 \kappa_\pm$.
For each sign, choose the branches of $x^{1/2}$ and $x^{1/4}$ determined by the indicated arguments on the sectors
\begin{eqn}
	\U_+ \coleq \set{ x ~\big|~ -\pi/5 < \arg x < 3\pi/5 }
\qqtext{and}
	\U_- \coleq \set{ x ~\big|~ -3\pi/5 < \arg x < \pi/5 }
\fullstop
\end{eqn}
The functions
\begin{eqn}
	Q_\pm (x) \coleq \kappa_\pm x^{1/2},
\qqquad
	\lambda_\pm (x) \coleq \varepsilon_\pm x^{1/4},
\qqquad
	z_\pm (x) \coleq \frac45 \varepsilon_\pm x^{5/4}
\fullstop{,}
\end{eqn}
satisfy $6Q_\pm^2 + x = 0$, $\lambda_\pm^2 = 12Q_\pm$, and $\del_x z_\pm = \lambda_\pm$.
Each $z_\pm$ maps $\U_\pm$ biholomorphically onto the upper half-plane.

With the chosen branches $Q_\pm$, the domains $\U_\pm$ are projections of Stokes sectors for the positive summation direction on the double cover $\tau^2 = -x/6$.
Borel summability and exact existence on these sectors were established in \cite[Proposition~2.14 and Corollary~2.15]{AlameddineMarchalNikolaevOrantin2025}.
The following proposition gives a direct proof using \autoref{1793810120014}.

\begin{proposition}[Exact perturbative solutions of Painlevé I]
\label{1793822824372}
For each sign, there is a unique formal solution $\hat{q}_\pm$ of \eqref{1793822831332} with leading-order term $Q_\pm$.
For every relatively compact open subset $\U_0 \Subset \U_\pm$, the formal solution $\hat{q}_\pm$ is uniformly Borel summable on $\U_0$.
There is a Borel disc $\S_0 = \S (r_0)$, which may depend on $\U_0$, such that $q_\pm = s [\hat{q}_\pm]$ is the unique holomorphic solution on $\U_0 \times \S_0$ admitting $\hat{q}_\pm$ as its uniform asymptotic expansion of factorial type.
For each sign, the solutions for different choices of $\U_0$ agree on the intersections of their domains.
\end{proposition}

\begin{proof}
Fix one sign and omit it from the notation.
Introduce $p = \hbar \del_x q$ to write \eqref{1793822831332} as the first-order system
\begin{eqn}
	\hbar \del_x \mtx{ q \\ p }
		= \mtx{ p \\ 6q^2 + x }
\fullstop
\end{eqn}
Its leading-order solution is $(Q,0)^T$, and the corresponding leading-order Jacobian is
\begin{eqn}
	J = \mtx{ 0 & 1 \\ 12Q & 0 }
\fullstop
\end{eqn}
Let
\begin{eqn}
	P \coleq \mtx{ 1 & 1 \\ \lambda & -\lambda },
\qquad
	P^{-1} = \frac{1}{2\lambda} \mtx{ \lambda & 1 \\ \lambda & -1 }
\fullstop{,}
\end{eqn}
and make the change of variables
\begin{eqn}
	\mtx{ q \\ p } = \mtx{ Q \\ 0 } + P \tilde{f}
\fullstop
\end{eqn}
Then $P^{-1} J P = \diag (1,-1) \lambda$.
For the calculation in the coordinate $z$, put
\begin{eqn}
	a \coleq \frac{\del_x \lambda}{2\lambda^2} = \frac{1}{10z},
\qqquad
	b \coleq \frac{\del_x Q}{2\lambda} = \frac{\kappa}{4\varepsilon} x^{-3/4},
\qqquad
	c \coleq \frac3{\lambda^2} = \frac{1}{4Q}
\fullstop
\end{eqn}
From now on, all functions of $x$ are composed with the inverse $x = x (z)$.
The substitution of \autoref{1793791517994} gives
\begin{eqns}
	\hbar \del_z \tilde{f}_1
		&= \tilde{f}_1 + c (\tilde{f}_1 + \tilde{f}_2)^2 - \hbar a (\tilde{f}_1 - \tilde{f}_2) - \hbar b
\fullstop{,}
\\
	-\hbar \del_z \tilde{f}_2
		&= \tilde{f}_2 + c (\tilde{f}_1 + \tilde{f}_2)^2 - \hbar a (\tilde{f}_1 - \tilde{f}_2) + \hbar b
\fullstop
\end{eqns}
In particular, the first-order correction of $\tilde{f}$ is $(b,-b)^T$.
As in \autoref{1793793108372}, make the further substitution
\begin{eqn}
	\tilde{f} = \hbar \left( \mtx{ b \\ -b } + g \right)
\fullstop
\end{eqn}
Since $\del_z b = -6ab$ and $4ab = 1/ (96 x^2)$, we obtain the normalised system \eqref{1793792948788} on the upper half-plane in the $z$-coordinate, with $V = \diag (\del_z,-\del_z)$ and
\begin{eqn}
	G_1 (z, \hbar, y) = G_2 (z, \hbar, y)
		= c (z) (y_1 + y_2)^2 - a (z) (y_1 - y_2) + 4a (z) b (z)
\fullstop
\end{eqn}

For $d > 0$, put $\V_+ \coleq \set{ z \in \CC ~\big|~ \Im z > d }$.
The real forward flows $\Phi_1^u (z) = z+u$ and $\Phi_2^u (z) = z-u$, for $u \geq 0$, preserve $\V_+$.
The functions $a,b,c$ are holomorphic on the upper half-plane and satisfy
\begin{eqn}
	a (z) \in O (|z|^{-1}),
\qqquad
	b (z) \in O (|z|^{-3/5}),
\qqquad
	c (z) \in O (|z|^{-2/5})
\fullstop
\end{eqn}
Consequently, $G$ is uniformly bounded for $z \in \V_+$ and $y$ in any fixed bounded polydisc.
Since $G$ is independent of $\hbar$, its perturbative expansion consists of $G$ itself in degree zero and satisfies the uniform estimates of factorial type on any Borel disc $\S (r)$.

For a given $\U_0 \Subset \U_\pm$, choose $d > 0$ smaller than $\inf_{x \in \overline{\U_0}} \Im z (x)$.
Then $z (\U_0) \Subset \V_+$.
Apply \autoref{1793810120014} on the common forward-invariant domain $\V_+$ to obtain a bounded holomorphic solution on $\V_+ \times \S_0$ for some Borel disc $\S_0 = \S (r_0)$.
Its restriction to $z (\U_0) \times \S_0$ is the uniform Borel resummation $g = s [\hat{g}]$ there.

Returning to the original variables gives
\begin{eqn}
	q = Q + \hbar (g_1 + g_2)
\qqtext{and}
	p = \hbar \del_x Q + \hbar \lambda (g_1 - g_2)
\fullstop
\end{eqn}
Reversing the substitutions shows that $(q,p)$ solves the first-order system.
Since $Q$, $\del_x Q$, and $\lambda$ are bounded on $\U_0$, these formulas transfer the uniform factorial asymptotics to $(q,p)$.
The formal solution is unique by \autoref{211209161918}.
The Nevanlinna-Sokal Theorem identifies $q$ with $s [\hat{q}]$.
For two choices $\U_0$ and $\U_0'$ for the same sign, let $r_0$ and $r_0'$ be the diameters of the corresponding Borel discs.
On $(\U_0 \cap \U_0') \times \S (\min \{r_0,r_0'\})$, both solutions have the same uniform asymptotic expansion of factorial type.
They agree there by \autoref{210617120300}, applied at each fixed $x$.
\end{proof}

For reference, put $q^\pto{0} = Q$ and $q^\pto{j} = 0$ for $j < 0$.
Then the formal recursion is
\begin{eqn}
	q^\pto{n}
		= \frac{\del_x^2 q^\pto{n-2} - 6 \sum_{r=1}^{n-1} q^\pto{r} q^\pto{n-r}}{12Q}
\qquad
	(n \geq 1)
\fullstop
\end{eqn}
In particular,
\begin{eqn}
	\hat{q} = Q - \frac{\hbar^2}{48x^2} + O (\hbar^4)
\fullstop
\end{eqn}
All odd coefficients vanish, since replacing $\hbar$ by $-\hbar$ preserves the equation and the selected leading-order term.


\clearpage
\bibliographystyle{src/declare-bibliography}
\bibliography{src/bibliography}

\end{document}